\documentclass{amsart}
\usepackage[utf8]{inputenc}
\usepackage[english,german]{babel}	

\usepackage{comment}
\usepackage{url} 
\usepackage{fixltx2e}

\usepackage{hyperref}
\usepackage{graphicx}
\usepackage{tabularx}

\usepackage{amsmath}
\usepackage{amssymb}			
\usepackage{amsfonts}
\usepackage{amsthm}
\usepackage{mathrsfs}
\usepackage{enumerate}

\usepackage{booktabs}
\usepackage[
singlelinecheck=true 
]{caption}

\usepackage{tikz}
  \usetikzlibrary{matrix}
  \usetikzlibrary{fit}
  \usetikzlibrary{backgrounds}
  \usetikzlibrary{arrows}
  \usetikzlibrary{shapes}

\theoremstyle{plain}
\newtheorem{thm}{Theorem}[section]
\newtheorem{lem}[thm]{Lemma}
\newtheorem{prop}[thm]{Proposition} 
\newtheorem{cor}[thm]{Corollary}
\newtheorem{con}[thm]{Conjecture}

\theoremstyle{definition}

\newtheorem{rem}[thm]{Remark}
\newtheorem{ex}[thm]{Example}

\def\QQ{{\mathbb Q}}
\def\CC{{\mathbb C}}

\def\ZZ{{\mathbb Z}}

\def\L{{\mathcal{L}}}

\def\epsilon{{\varepsilon}}
\def\phi{{\varphi}}

\title{Computing $\L$-invariants via the Greenberg-Stevens formula}
\author{Samuele Anni, Gebhard Böckle, Peter Mathias Gräf, \'{A}lvaro Troya}
\address{University of Luxembourg, RMATH, 6 avenue de la Fonte, 4364 Esch-sur-Alzette, Luxembourg}
\email{samuele.anni@gmail.com}

\address{Universität Heidelberg, IWR, Im Neuenheimer Feld 205, 69120 Heidelberg, Germany}
\email{gebhard.boeckle@iwr.uni-heidelberg.de} 
\email{peter.graef@iwr.uni-heidelberg.de}
\email{troya@stud.uni-heidelberg.de}

\date{\today}

\begin{document}

\selectlanguage{english}
\maketitle

\vspace{-3em}
\begin{abstract}
In this article, we describe how to compute slopes of $p$-adic $\L$-invariants of arbitrary weight and level by means of the Greenberg-Stevens formula. Our method is based on the work of Lauder and Vonk on computing the reverse characteristic series of the $U_p$-operator on overconvergent modular forms. Using higher derivatives of this characteristic series, we construct a polynomial whose zeros are precisely the $\L$-invariants appearing in the corresponding space of modular forms with fixed sign of the Atkin-Lehner involution at $p$. In addition, we describe how to compute this polynomial efficiently. In the final section, we give computational evidence for relations between slopes of $\L$-invariants for small primes.
\end{abstract}

\section{Introduction}

Let $f$ be a newform of even weight $k\geq 2$ for $\Gamma_0(pN)$, where $p$ is prime and $(N,p)=1$. Let $\chi$ be a Dirichlet character of conductor prime to $pN$ with $\chi(p)= 1$. By the work of Mazur and Swinnerton-Dyer, see for example \cite{mtt}, there exists a $p$-adic {$L$-function} $L_p(f,\chi,s)$ attached to $f$ that interpolates the algebraic parts $L^{\text{alg}}(f,\chi,j)$ for $j\in\{1,\dots,k-1\}$ of the special values of the classical $L$-function attached to $f$. The $p$-adic $\mathcal{L}$-invariant $\mathcal{L}_p(f)\in\CC_p$ attached to $f$ satisfies 
\[
L'_p(f,\chi,\tfrac{k}{2})=\mathcal{L}_p(f)\cdot L^{\mathrm{alg}}(f,\chi,\tfrac{k}{2})
\]
and it depends only on the local Galois representation attached to $f$. Mazur, Tate and Teitelbaum conjectured that such invariant should exist. Afterwards, several possible candidates were proposed. The fundamental breakthrough in relating these candidates to one another and proving the above formula, due to Greenberg and Stevens \cite{gs}, is the connection between $\mathcal{L}_p(f)$ and the (essentially unique) $p$-adic family of eigenforms passing through $f$. If we denote this family by 
\[
f_\kappa=\sum_{n=1}^\infty a_n(\kappa)q^n \quad \text{with} \quad f_k=f,
\]
where the coefficients $a_n(\kappa)$ are rigid analytic functions on a disc containing $k$ in the weight space $\mathrm{Hom}_{\mathrm{cont}}(\ZZ_p^\times,\mathbb{G}_m)$, then the relation is given as
\[
\mathcal{L}_p(f)=-2\operatorname{\mathrm{dlog}}(a_p(\kappa))|_{\kappa=k}.
\]
A more detailed account of the history of $\mathcal{L}$-invariants can be found in \cite{col10}.

Recently, there has been a lot of interest in computing $p$-adic $\mathcal{L}$-invariants. In \cite{pg}, the third author conjectured some relations between (slopes of) $\mathcal{L}$-invariants of different levels and weights for $p=2$. In \cite{berg}, Bergdall explains relations between slopes of $\mathcal{L}$-invariants and the \emph{size} of the $p$-adic family passing through the given newform.
On a different note, a main motivation to understand and compute $\mathcal{L}$-invariants arises from the study of coefficient fields of classical newforms. In particular, Buzzard asked whether there exists a bound $B_{N,p}$ for all Hecke eigenforms $f$ of level $\Gamma_1(N)\cap \Gamma_0(p)$ and any weight $k\geq 2$ such that $[\QQ_{f,p} : \QQ_p]\leq B_{N,p}$, where $\QQ_{f,p}$ is the coefficient field of $f$ completed at a prime dividing $p$, see \cite[Question 4.4]{questions_buzzard}. 
For $N = 1$ and $p \leq 7$ Chenevier has shown (\cite[Corollaire~p.3]{Chenevier}) that $\QQ_{f,p}$ is either 
$\QQ_p(a_p(f))$ if $f$ is old at $p$, or $\QQ_p (\mathcal{L}_p(f))$ if $f$ is new at $p$, where $a_p(f)$ is the $p$-th coefficient of the $q$-expansion of $f$ at infinity. Therefore, Buzzard's bound $B_{N,p}$ would constrain the degrees of $\mathcal{L}_p(f)$ and $a_p(f)$ over $\QQ_p$ for all $f$ as long as $N$ and $p$ are fixed. 

In this article, we describe a procedure for computing $\mathcal{L}$-invariants via the Greenberg-Stevens formula, building on the work of Lauder and Vonk on  the action of the $U_p$-operator on  overconvergent modular forms. In \cite[Section 3.3.4]{lau} Lauder describes how to compute the $\mathcal{L}$-invariant from the first derivative of the (inverse)  characteristic series of the $U_p$-operator if there is a unique split multiplicative cusp form in the given weight, using a formula of Coleman, Stevens and Teitelbaum \cite{cst}. The main aim of this article is to extend this method to weights $k$ in which there is more than one split multiplicative cusp form. For this purpose, we explain how to efficiently compute higher derivatives of the characteristic series and how the $\mathcal{L}$-invariants can be read off from these derivatives. In order for this approach to work and not increase too much the precision required to perform the computations, one first needs to decompose the space of $p$-newforms by their Atkin-Lehner eigenvalue at $p$. Since the needed precision grows with the dimension of these subspaces, we give explicit dimension formulas, analogous to the ones presented in \cite{mar}.\par
Let us remark that the procedure presented in \cite{pg} computes more data than the present paper, as  the $\mathcal{L}$-invariants in \cite{pg} are realized as the eigenvalues of  a "Hecke operator", called the $\mathcal{L}$-operator, whose matrix is computed explicitly. However, the procedure is naturally more restricted, since it involves the Jacquet\--Langlands correspondence and, therefore, it needs an auxiliary prime in the level. Another approach (unpublished) due to R. Pollack is to compute the $\mathcal{L}$-invariants directly in terms of $p$-adic and classical $L$-values via modular symbols.  Yet another approach is due to X. Guitart and M. Masdeu, see \cite{guimas}.

The layout is as follows. In Section~\ref{sec2} we recall briefly Coleman classicality and the main result of \cite{lau}, which are going to be used in several central parts of this article. In Section~\ref{sec3}, we describe how to generalize a formula of Coleman, Stevens and Teitelbaum \cite{cst} proving the existence of a polynomial $\mathcal{Q}_{p,k}$, built from higher derivatives of the reverse characteristic series of the $U_p$ operator on the space of overconvergent $p$-adic cusp forms of level $\Gamma_0(pN)$, and whose roots are precisely the $\mathcal{L}$-invariants of level $\Gamma_0(pN)$. In Sections~\ref{sec4} and~\ref{sec5} we show how to compute this polynomial extending a method of Lauder \cite{lau} and we give dimension formulae for the relevant spaces of classical modular forms.
In Section~\ref{sec6} we present the data collected together with observations on the slopes.

All computations have been done using a \texttt{magma} implementation of the algorithms. The code relies on algorithms of Lauder and Vonk and is available on request.

\proof[Acknowledgements]
This work was partially supported by the DFG Forschergruppe 1920, the DFG Priority Program SPP 1489 and the Luxembourg FNR. 
We would also thank Alan Lauder, Robert Pollack and Jan Vonk for helpful remarks and suggestions. Finally, we would like to thank John Cremona for providing access to the servers of the Number Theory Group at the Warwick Mathematics Institute.

\section{Classical and overconvergent modular forms}
\label{sec2}
Throughout this article, let $p$ be prime and let $N$ be a positive integer coprime to $p$. Let $\mathcal{W}$ denote the even $p$-adic weight space. Thus, the $\CC_p$-points of $\mathcal{W}$ are continuous characters $\kappa\colon\ZZ_p^\times\rightarrow\CC_p^\times$ with $\kappa(-1)=1$. For each $\kappa\in\mathcal{W}$, let $\mathcal{S}^\dagger_\kappa(\Gamma_0(pN))$ denote the space of overconvergent $p$-adic cusp forms of weight $\kappa$, see \cite{cole}. This space is equipped with the action of the compact operator $U_p$. 
For each integer $k$ we realize the space of classical cuspforms $\mathcal{S}_k(\Gamma_0(pN))$ as a $U_p$-stable subspace of $\mathcal{S}^\dagger_k(\Gamma_0(pN))$ by considering the associated weight $z\mapsto z^k$ in $\mathcal{W}$. Let
\[
\begin{aligned}
&P(k,t):=\det(1-t\, U_p\mid \mathcal{S}_k(\Gamma_0(pN))),\\
&P^\dagger(\kappa,t):=\det(1-t\,U_p\mid \mathcal{S}^\dagger_\kappa(\Gamma_0(pN))),
\end{aligned}
\]
be the reverse characteristic polynomial (respectively series) of $U_p$. We denote by $\mathcal{O}(\mathcal{W})\{\!\{ t\}\!\}$ the ring of entire functions with coefficients in $\mathcal{O}(\mathcal{W})$. Then, if we write 
\[
P^\dagger(\kappa,t)=1+\sum_{n=1}^\infty b_n(\kappa)t^n\in \mathcal{O}(\mathcal{W})\{\!\{ t\}\!\},
\]
by \cite[Appendix I]{cole} each function $\kappa\mapsto b_n(\kappa)$ is defined by a power series with coefficients in $\ZZ_p$. More precisely, if we write $\mathcal{W}=\bigsqcup_\epsilon\mathcal{W}_\epsilon$, where the union is formed over the even characters $\epsilon\colon (\ZZ/2p\ZZ)^\times\rightarrow\CC_p^\times$, then we have $\kappa\in\mathcal{W}_\epsilon$ if and only if the restriction of $\kappa$ to the torsion subgroup of $\ZZ_p^\times$ is $\epsilon$. For a fixed topological generator $\gamma$ of $1+2p\ZZ_p$, each $\mathcal{W}_\epsilon$ is an open $p$-adic unit disc with coordinate $w_\kappa=\kappa(\gamma)-1$. Then, for each such $\epsilon$, there exists 
\[
P^{\dagger,\epsilon}(w,t)=1+\sum_{n=1}^\infty b_n^\epsilon (w)t^n\in\ZZ_p[\![w]\!]\{\!\{t\}\!\}
\]
such that $P^\dagger(\kappa,t)=P^{\dagger,\epsilon}(w_\kappa,t)$ for $\kappa\in\mathcal{W}_\epsilon$. In the sequel, we surpress $\epsilon$ in the notation. The following theorem, due to Coleman, links $P(k,t)$ and $P^\dagger(k,t)$.  For a detailed account in the cuspidal situation, see \cite[Section 2.1.2]{bela}.

\begin{thm}[{Coleman classicality, \cite[Theorem 6.1]{class}}]\label{classicality}
Let $\alpha<k-1$. Then the roots of $P^\dagger(k,\cdot)$ of valuation $\alpha$ are precisely the roots of $P(k,\cdot)$ of valuation $\alpha$. Consequently, if $m\leq k-1$ then
\[
P(k,t)\equiv P^\dagger(k,t) \mod p^m.
\]
\end{thm}

\begin{rem}
\label{algo}
For $p\geq 5$, the algorithm of Lauder presented in \cite{lau} computes $P^\dagger(k,t) \mod p^m$ for given $k$ and $m$  and runs in polynomial time with respect to $p, N$ and $m$, and linear time in $\log(k)$. This algorithm was extended by Vonk in \cite{vonk} to include the primes $p=2,3$. These algorithms give us the input for our subsequent computations. 
\end{rem}

\section{$\mathcal{L}$-invariants and derivatives of the characteristic series of $U_p$}
\label{sec3}
For a fixed positive even integer $k$, in \cite{cst} Coleman, Stevens and Teitelbaum showed that whenever there is a unique split multiplicative cusp form $f$ in $S_{k}(\Gamma_0(pN))$, it is possible to compute its $\mathcal{L}$-invariant as   
$$\L_p(f) = -2a_p(f)\dfrac{\partial_\kappa P^\dagger(k,t)}{\partial_t P^\dagger(k,t)}\bigg\vert_{t=a_p(f)^{-1}}$$
where $a_p(f)$ is the $p$\--th coefficient of the $q$-expansion of the newform. Clearly, the formula above works only in the special case where the space of split multiplicative forms in $S_{k}(\Gamma_0(pN))$ is $1$-dimensional, as otherwise the denominator has a zero at $t=a_p(f)^{-1}$.  In \cite{lau}, Lauder uses the formula in this special case to compute the $\mathcal{L}$-invariant. 

The aim of this section is to remove the restriction on the dimension. If we denote by $\mathcal{S}_k(\Gamma_0(pN))^{p-\mathrm{new}}$ the subspace of cuspforms which are new at $p$, then our approach can be described as follows. First we split $\mathcal{S}_k(\Gamma_0(pN))^{p-\mathrm{new}}$  into eigenspaces for the Atkin-Lehner operator at $p$. Then use higher derivatives to distinguish among the zeros of $P^\dagger(k,t)$ which correspond to different eigenforms, i.e. to separate the $p$\--adic families passing through the different eigenforms.\par
\vspace{\baselineskip}
We begin by recalling the definition of $p$-adic $\mathcal{L}$-invariants. Note that for each $f\in \mathcal{S}_k(\Gamma_0(pN))^{p-\mathrm{new}}$, we have
\[
U_pf=-p^{(k-2)/2}W_pf,
\]
where $W_p$ denotes the Atkin-Lehner involution acting on $\mathcal{S}_k(\Gamma_0(pN))^{p-\mathrm{new}}$. Let $f\in\mathcal{S}_k(\Gamma_0(pN))^{p-\mathrm{new}}$ be an eigenform for the Hecke operators away from $N$. Then $f$ has an associated  $\mathcal{L}$-invariant $\mathcal{L}_p(f)\in\CC_p$. By a generalization of the Greenberg-Stevens formula, see \cite{col10}, $\mathcal{L}_p(f)$ is given as follows. The eigensystem (away from $N$) attached to $f$ defines a classical point on the eigencurve $\mathcal{C}_N$ of level $\Gamma_0(p)\cap\Gamma_1(N)$, see \cite{buz}. There is a $p$-adic Coleman family $f_\kappa=\sum_{n=1}^\infty a_n(\kappa)q^n$ through $f$ (i.e. an irreducible component of a small affinoid neighbourhood of the point attached to $f$ in $\mathcal{C}_N$ that maps isomorphically onto an open affinoid subdomain of the weight space) of constant slope since the weight map is \'etale at the point attached to $f$, see \cite[Proposition 2.6]{berg}. In particular, via this $p$-adic family, the rigid morphism $a_p(\kappa)$ on $\mathcal{C}_N$ defines a morphism on the affinoid subdomain of the weight space. The $\mathcal{L}$-invariant attached to $f$ is then given by
\[
\mathcal{L}_p(f)= -2\operatorname{\mathrm{dlog}}(a_p(\kappa))|_{\kappa=k}.
\]
Let us point out that, for $\kappa$ sufficiently close to $k$, the eigenvalue $a_p(\kappa)$ in level $\Gamma_0(p)\cap\Gamma_1(N)$ appears in fact in level $\Gamma_0(pN)$; i.e. on the space $\mathcal{S}_\kappa^\dagger(\Gamma_0(pN))$. The expositions in \cite{berg} and \cite{bela} provide a nice summary of the constructions that are relevant in our setting.

\smallskip

Before we can study the relation between $\mathcal{L}$-invariants and the characteristic series of $U_p$, we need some preparations.

\begin{lem}
\label{eigen}
Let $\epsilon$ be an eigenvalue of the Atkin-Lehner operator $W_p$ acting on $\mathcal{S}_k(\Gamma_0(pN))$. Let $a^\epsilon_{p,k}=-\epsilon p^{(k-2)/2}$. Then 
$$\mathcal{S}_k(\Gamma_0(pN))^{p-\mathrm{new}, \epsilon}=\mathcal{S}_k(\Gamma_0(pN))^{U_p=a^\epsilon_{p,k}},$$
where $\mathcal{S}_k(\Gamma_0(pN))^{p-\mathrm{new},\epsilon}$ denotes the subspace of $\mathcal{S}_k(\Gamma_0(pN))^{p-\mathrm{new}}$
on which $W_p$ has eigenvalue $\epsilon$, and $\mathcal{S}_k(\Gamma_0(pN))^{U_p=a^\epsilon_{p,k}}$ is the subspace of $\mathcal{S}_k(\Gamma_0(pN))$ where the $U_p$ operator acts as multiplication by $a^\epsilon_{p,k}$.
\end{lem}

\proof By definition $\mathcal{S}_k(\Gamma_0(pN))^{p-\mathrm{new}, \epsilon} \subseteq \mathcal{S}_k(\Gamma_0(pN))^{U_p=a^\epsilon_{p,k}}$. In order to show 
the other inclusion, let us first remark that if a cuspform in $\mathcal{S}_k(\Gamma_0(pN))^{U_p=a^\epsilon_{p,k}}$ is new at $p$, then clearly it is an eigenform for $W_p$ with eigenvalue $\epsilon$. 
Let us now look at the space of $p$-oldforms $\mathcal{S}_k(\Gamma_0(pN))^{p-\mathrm{old}}$. Each eigenvalue of $T_p$ acting on $\mathcal{S}_k(\Gamma_0(N))$ gives rise to two $U_p$-eigenvalues on $\mathcal{S}_k(\Gamma_0(pN))^{p-\mathrm{old}}$ and by definition all eigenvalues of $p$-oldforms arise in this way. Let $\alpha$ be a $T_p$-eigenvalue. The corresponding $U_p$-eigenvalues are the roots of $x^2-\alpha x+p^{k-1}$. Now, by a deep theorem of Deligne, these $U_p$-eigenvalues have complex absolute value $p^{(k-1)/2}$, so they cannot be $a^\epsilon_{p,k}$.
\endproof

Let us remark  that $\mathcal{S}_k(\Gamma_0(pN))^{p-\mathrm{new}}$ has a basis of eigenforms for the Hecke operators away from $N$. These split into two orbits with respect to the eigenvalues of the Atkin-Lehner involution. 

From now on we fix an eigenvalue $\epsilon$ of $W_p$ and set  $d_\epsilon=\dim \mathcal{S}_k(\Gamma_0(pN))^{p-\mathrm{new}, \epsilon}$. 
Let $f_1,\dots,f_{d_\epsilon}$ denote the basis of $\mathcal{S}_k(\Gamma_0(pN))^{p-\mathrm{new}, \epsilon}$. For each $i\in\{1,\dots,d_\epsilon\}$, we denote by $a_i(\kappa)$ the $p$-th coefficient of the Coleman family passing through $f_i$ as above.
\begin{prop}
\label{order_vanishing}
Let $\epsilon$ be an eigenvalue of the Atkin-Lehner operator $W_p$ acting on $\mathcal{S}_k(\Gamma_0(pN))^{p-\mathrm{new}}$, then 
$$P^\dagger(k,t)=(1-a^\epsilon_{p,k}t)^{d_\epsilon}C(k,t),$$ where $C(k, (a^\epsilon_{p,k})^{-1})\neq 0$.
\end{prop}

\proof
For $\alpha <k-1$, by Coleman classicality, Theorem~\ref{classicality}, the set of roots of $P^\dagger(k,t)$ of slope $\alpha$ is the set of roots of $P(k,t)$ of slope $\alpha$. Therefore, Lemma~\ref{eigen} implies the claim.\endproof

\begin{prop}
\label{factorweight}
There exists a small neighbourhood $V\subset\mathcal{W}$ of $k$ such that in $\mathcal{O}(V)\{\!\{ t\}\!\}$ we have
$$P^\dagger(\kappa,t)=\prod_{i=1}^{d_\epsilon}(1-a_i(\kappa)t) \cdot C(\kappa,t).$$
where the function $C(\kappa,t)$ specializes to the corresponding function in Proposition~\ref{order_vanishing} for $\kappa=k$.
\end{prop}
\begin{proof}
By  \cite[Proposition II.1.12]{belacourse}, we find a small neighbourhood $V$ of $k$ such that in $\mathcal{O}(V)\{\!\{ t\}\!\}$ we have
\[
P^\dagger(\kappa,t)=F(\kappa,t)\cdot G(\kappa,t),
\]
where $F(\kappa,t)\in\mathcal{O}(V)[t]$ whose roots are precisely the roots of $P^\dagger(\kappa,t)$ of valuation less or equal to $k-1$ (with the same multiplicities), $F(\kappa,0)=1$  and $G(\kappa,t)\in \mathcal{O}(V)\{\!\{ t\}\!\}$. After possibly shrinking $V$, the Coleman families through the $f_i$ do not intersect and $a_i(\kappa)\in\mathcal{O}(V)$. Then, by construction, these are zeros of $F(\kappa,t)$ whose multiplicities add up to $d_\epsilon$. Now factoring $F(\kappa,t)$ gives the result.
\end{proof}

In the sequel, for $\bullet \in \{\kappa, t\}$, let us denote by $\partial_\bullet^n P^\dagger(k,t)$ the $n$-th partial derivative of $P^\dagger(\kappa,t)$ with respect to $\bullet$ at the point $(k,t)$, viewed as an element in $\ZZ_p\{\!\{ t\}\!\}$.
\begin{cor} 
\label{higherder}
For $n\leq d_\epsilon$ and $\bullet \in \{\kappa, t\}$, we have
\[
\partial_\bullet^n P^\dagger(k,t)= (1-a_{p, k}^{\epsilon}t)^{d_\epsilon-n}\cdot C^{(n)}_\bullet(k,t)
\]
where
\[
C^{(n)}_\bullet(k,(a^\epsilon_{p,k})^{-1})=n\,!\cdot C(k,(a^\epsilon_{p,k})^{-1})\sum_{1\leq m_1<\dots<m_n\leq d_\epsilon}\,\prod_{j=1}^n\partial_\bullet(1-a_{m_j}(\kappa)t)\bigg\vert_{\kappa=k, t=(a^\epsilon_{p,k})^{-1}}.
\]
\end{cor}
\proof
By Proposition \ref{factorweight}, we have
$$P^\dagger(\kappa,t)=\prod_{i=1}^{d_\epsilon}(1-a_i(\kappa)t) \cdot C(\kappa,t)$$
on a small neighbourhood of $k$. Thus, the $n$-th partial derivative of $P^\dagger$ for $\bullet \in \{\kappa, t\}$ satisfies
$$
\partial_\bullet^n P^\dagger(\kappa,t)= \sum_{i_1+i_2+\dots+i_{d_\epsilon+1}=n}{n\choose
i_1,i_2,\dots,i_{d_\epsilon+1}}\cdot \prod_{j=1}^{d_\epsilon}\partial_\bullet^{i_j}(1-a_j(\kappa)t) \cdot \partial_\bullet^{i_{d_\epsilon+1}}C(\kappa,t).
$$
\noindent For $n\leq d_\epsilon$, let us split the sum into $\Sigma_{\leq 1}(\kappa,t)$ and $\Sigma_{>1}(\kappa,t)$ as follows:
\begin{align*}
\Sigma_{\leq 1}(\kappa,t) &= \sum_{\substack{i_1+i_2+\dots+i_{d_\epsilon+1}=n, \\ i_j \leq 1 \, \forall\, j }}{n\choose i_1,i_2,\dots,i_{d_\epsilon+1}}\cdot \prod_{j=1}^{d_\epsilon}\partial_\bullet^{i_j}(1-a_j(\kappa)t)\cdot \partial_\bullet^{i_{d_\epsilon+1}}C(\kappa,t), \\
\Sigma_{>1}(\kappa,t)  &= \sum_{\substack{i_1+i_2+\dots+i_{d_\epsilon+1}=n,\\ i_j > 1 \text{ for some } j }}{n\choose 
i_1,i_2,\dots,i_{d_\epsilon+1}}\cdot \prod_{j=1}^{d_\epsilon}\partial_\bullet^{i_j}(1-a_j(\kappa)t)\cdot \partial_\bullet^{i_{d_\epsilon+1}}C(\kappa,t).
\end{align*}
We claim that $\Sigma_{>1}(k,t) = (1-a_{p,k}^\epsilon t)^{d_\epsilon-n+1}\cdot C^{(n)}_{>1}(k,t)$ for some $C^{(n)}_{>1}(k,t)$.\\ 
If $i_\ell>1$ for some $\ell \in \{1,\dots,d_\epsilon+1\}$,  we must have $i_j=0$ for more than $d_\epsilon-n+1$ many $j$'s, since all $i_j$ are non-negative, their sum equal to $n$ and the cardinality of the range is $d_\epsilon+1$. Specializing to $\kappa=k$, this immediately implies the result.
We can proceed in a similar fashion to show that the terms with $i_{d_\epsilon+1} = 1$ in $\Sigma_{\leq 1}$ satisfy an analogous formula. Therefore, we are left with the term
$$
n\,!\cdot C(\kappa,t)\sum_{\substack{i_1+\dots+i_{d_\epsilon}=n,\\ i_j \leq 1 \, \forall\, j }}\,\prod_{j=1}^{d_\epsilon}\partial_\bullet^{i_j}(1-a_j(\kappa)t).
$$
Let us observe that after separating the terms with $i_j=0$ at $\kappa=k$ and evaluating at $t=(a^\epsilon_{p,k})^{-1}$ we obtain
\begin{align*}
n\,!\cdot C(k,t)\cdot (1-a_{p, k}^{\epsilon}t)^{d_\epsilon-n} \sum_{1\leq m_1<\dots<m_n\leq d_\epsilon}\,\prod_{j=1}^n\partial_\bullet(1-a_{m_j}(\kappa)t)\bigg\vert_{\kappa=k}.
\end{align*}
Putting these terms together completes the proof. 
\endproof

\begin{rem}
In the spirit of the above proposition, one could also analyse the behaviour of mixed derivatives with respect to $\kappa$ and $t$. However, for our applications only the above derivatives are relevant.
\end{rem}

Finally we can state the main result of this section:

\begin{thm}
\label{poly}
Let $\epsilon$ be an eigenvalue of the Atkin-Lehner operator $W_p$ acting on $\mathcal{S}_k(\Gamma_0(pN))^{p-\mathrm{new}}$, then the monic polynomial $\mathcal{Q}_{p,k}^{\epsilon}(x)=\sum_{n=0}^{d_\epsilon} c_{n} x^{d_\epsilon-n}\in \QQ_p[x]$,
where
$$c_n:=2^n (a_{p, k}^{\epsilon})^n{d_\epsilon\choose n}\dfrac{\partial^n_\kappa P^\dagger(\kappa,t)}{\partial^n_tP^\dagger(\kappa,t)}\bigg\vert_{\kappa=k,\;t=(a_{p, k}^{\epsilon})^{-1}},$$
satisfies
$$\mathcal{Q}_{p,k}^{\epsilon}(x)=\prod_{i=1}^{d_\epsilon}(x-\mathcal{L}_p(f_i)).$$
\end{thm}
\proof
Taking the ratio of the $n$-th partial derivatives of $P^\dagger(\kappa, t)$ and applying Corollary \ref{higherder} we obtain 
\begin{align*}
\dfrac{\partial^n_\kappa P^\dagger(\kappa,t)}{\partial^n_tP^\dagger(\kappa,t)}\bigg\vert_{\kappa=k,\;t=(a_{p, k}^{\epsilon})^{-1}}
    &=\dfrac{\sum_{1\leq m_1<\dots<m_n\leq d_\epsilon}\,\prod_{j=1}^n\partial_\kappa(1-a_{m_j}(\kappa)t)}
    {\sum_{1\leq m_1<\dots<m_n\leq d_\epsilon}\,\prod_{j=1}^n\partial_t(1-a_{m_j}(\kappa)t)}\bigg\vert_{\kappa=k,\;t=(a_{p, k}^{\epsilon})^{-1}}\\
    &=\dfrac{\sum_{1\leq m_1<\dots<m_n\leq d_\epsilon}\,(-a_{p, k}^{\epsilon})^{-n}\prod_{j=1}^n a'_{m_j}(k)}
    {\sum_{1\leq m_1<\dots<m_n\leq d_\epsilon}\,(-a_{p, k}^{\epsilon})^{-n}}\\
    &={d_\epsilon \choose n}^{-1}(a_{p, k}^{\epsilon})^{-n}\sum_{1\leq m_1<\dots<m_n\leq d_\epsilon}\,\prod_{j=1}^n \dfrac{a'_{m_j}(k)}{a_{p, k}^{\epsilon}}.
\end{align*}
Multiplying both sides by $2^n (a_{p, k}^{\epsilon})^{n}{d_\epsilon\choose n}$ yields
\begin{align*}
    c_n&=2^n (a_{p, k}^{\epsilon})^{n}{d_\epsilon\choose n}\dfrac{\partial^n_\kappa P^\dagger(\kappa,t)}{\partial^n_tP^\dagger(\kappa,t)}\bigg\vert_{\kappa=k,\,t=(a_{p, k}^{\epsilon})^{-1}}\\
    &=\sum_{1\leq m_1<\dots<m_n\leq d_\epsilon}\,\prod_{j=1}^n (2\operatorname{\mathrm{dlog}}(a_{m_j}(\kappa))|_{\kappa=k})
\end{align*}
The last term is precisely the $(d_\epsilon -n)$\--th coefficient of the polynomial
$$\prod_{i=1}^{d_\epsilon}\left(x-\mathcal{L}_p(f_i))\right. , $$
which completes the proof. 
\endproof

\begin{rem}
\label{nonzero}
Note that in the above proof we used that $C_t^{(n)}(k,(a^\epsilon_{p,k})^{-1})\neq 0$. The analogous statement for $C_\kappa^{(n)}(k,(a^\epsilon_{p,k})^{-1})$ is equivalent to  all of the $\mathcal{L}$-invariants $\mathcal{L}(f_i)$ being non-zero. While all our computational results support this claim, we are not aware of a proof of this statement. However, in the special case where $f_i$ is of weight $2$ and corresponds to an elliptic curve, this follows from a deep result in transcendental number theory stating that the Tate parameter is transcendental, see \cite{transc}.
\end{rem}

\section{Dimension formulae}
\label{sec4}
In this section we show how to compute the dimension $d_\epsilon$ of the cuspform space $\mathcal{S}_k(\Gamma_0(pN))^{p-\mathrm{new}, \epsilon}$ which appears in Section~\ref{sec3}. Let us recall that $N$ is a positive integer coprime to $p$. As before, we will denote by $W_p$ the $p$-th Atkin-Lehner involution and whenever $W_p$ is an operator on a vector space $V$, we will denote by $\mathrm{tr}_{V}W_p$ its trace on $V$.

\begin{prop}
The trace of the Atkin-Lehner operator $W_p$ satisfies 
 $$\mathrm{tr}_{\mathcal{S}_k(\Gamma_0(pN))^{p-\mathrm{new}}}W_p=\mathrm{tr}_{\mathcal{S}_k(\Gamma_0(pN))}W_p.$$
\end{prop}

\proof The cuspform space $\mathcal{S}_k(\Gamma_0(pN))$ decomposes into a $p$-new and a $p$-old component 
$\mathcal{S}_k(\Gamma_0(pN))=\mathcal{S}_k(\Gamma_0(pN))^{p-\mathrm{new}}\oplus\mathcal{S}_k(\Gamma_0(pN))^{p-\mathrm{old}}$.
The statement is then equivalent to say that $\mathrm{tr}_{\mathcal{S}_k(\Gamma_0(pN))^{p-\mathrm{old}}}W_p=0$ and this follows from 
\cite[Lemma~26]{Wp}.
\endproof

In the case of $N$ squarefree, building on previous works of Yamauchi, Skoruppa and Zagier, Martin gave a formula~\cite[Equation~(1.6)]{mar} for the trace of the Atkin-Lehner operator $W_p$ on $\mathcal{S}_k(\Gamma_0(pN))$. In the simple case where $N=q>3$ is a prime and $p>3$, we have
$$\mathrm{tr}_{\mathcal{S}_k(\Gamma_0(pq))}W_p=\frac{1}{2}(-1)^k a(p,q) h(\QQ(\sqrt{-p})) (1+\left(\frac{\Delta_p}{q}\right))+\delta_{k,2},$$
where $a(p,q)$ is $1,4$ or $2$ for $p$ congruent to $1$ or $5$, $3$, $7$ modulo $8$ respectively, $h(\QQ(\sqrt{-p}))$ denotes the class number of $\QQ(\sqrt{-p})$ and $\Delta_p$ its discriminant, the Legendre symbol for $\Delta_p$ and $q$ is denoted by $\left(\frac{\Delta_p}{q}\right)$, and $\delta_{k,2}$ is $1$ if $k=2$ and $0$ otherwise.

By simple algebraic manipulations, we have the following
\begin{cor}
\label{formula_dim}
Let $\epsilon$ be an eigenvalue of the Atkin-Lehner operator $W_p$ and set $d:=\dim\mathcal{S}_k(\Gamma_0(pN))^{p-\mathrm{new}}$. Then
$$d_\epsilon=\frac{1}{2}\left(d+\epsilon\; \mathrm{tr}_{\mathcal{S}_k(\Gamma_0(pN))}W_p\right)$$
\end{cor}

Let us remark that the dimension of $\mathcal{S}_k(\Gamma_0(pN))^{p-\mathrm{new}}$ can be obtained by recursively computing the dimension of $\mathcal{S}_k(\Gamma_0(pN))^{p-\mathrm{old}}$.  In Table~\ref{table:0} there are dimensions of $\mathcal{S}_k(\Gamma_0(pN))^{p-\mathrm{new}}$ and of $\mathcal{S}_k(\Gamma_0(pN))^{p-\mathrm{new}, \epsilon}$ for some levels and weights, where, for simplicity, we set $\mathrm{Tr}:=\mathrm{tr}_{\mathcal{S}_k(\Gamma_0(pN))}W_p$.
\setcounter{table}{-1}
\begin{table}[h]
\begin{tabular}{c|c|c|c|c|c|c}
$p$ & $N$ & $k$ & $\mathrm{Tr}$ &$d$ & $d_{+1}$ & $d_{-1}$ \\
\hline
$2$ & $1$ & $20$ & $0$ & $2$ & $1$ & $1$ \\
$2$ & $1$ & $40$ & $1$ & $3$ & $2$ & $1$ \\
$3$ & $1$ & $20$ & $1$ & $3$ & $2$ & $1$ \\
$3$ & $1$ & $40$ & $0$ & $6$ & $3$ & $3$ \\
$13$ & $1$ & $20$ & $1$ & $19$ & $10$ & $9$ \\
$13$ & $1$ & $40$ & $1$ & $39$ & $20$ & $19$ \\
$13$ & $2$ & $20$ & $1$ & $57$ & $29$ & $28$ \\
$13$ & $2$ & $40$ & $1$ & $117$ & $59$ & $58$  \\
$7$ & $11$ & $200$ & $2$ &$1194$ & $598$ & $596$\\
$7$ & $11$ & $400$ & $2$ &$2394$ & $1198$ & $1196$ 
\end{tabular}
\caption{}
\label{table:0}
\end{table}

\section{Computing $\mathcal{Q}_{p,k}^{\epsilon}$}
\label{sec5}
The main difficulty in computing  $\mathcal{Q}_{p,k}^{\epsilon}$ is given by the partial derivatives, as shown in Theorem \ref{poly}. For the denominators, this is straightforward by Remark \ref{algo}, since we only use formal differentiation with respect to $t$. However, the computation of the numerators is more involved. 

In \cite[Lemma~3.10]{lau} Lauder explains a method using finite differences for the first derivative. The aim of this section is to generalize this method to derivatives of higher order.
The study of finite differences in this context is due to Gouv\^ea and Mazur, see \cite{gm}.

Let us first observe that it is enough to work only at classical points since, by definition, we have
$$\partial^n_\kappa P^\dagger(k,t)=\lim_{m \to \infty}\dfrac{\partial
    ^{n-1}_\kappa P^\dagger(k + (p-1)p^m, t) - \partial^{n-1}_\kappa P^\dagger(k,t)}{(p-1)p^m}.
$$
The finite differences used to approximate $\partial^n_\kappa P^\dagger(k,t)$ are defined as follows. Let $s_m=(p-1)p^m$ and set 
\[
\partial^n_\kappa P^\dagger(k,t)_m := s_m^{-n}\sum_{j=0}^n(-1)^j{n\,\choose j} P^\dagger(k + (n-j)s_m,t).
\]

\begin{thm}
\label{prec_deriv}
We have 
\[
\partial^{n}_\kappa P^\dagger(k,t)_{m+1} \equiv \partial^{n}_\kappa P^\dagger(k,t)_m \mod p^{m+1}.
 \]
Therefore, $\partial^{n}_\kappa P^\dagger(k,t)\equiv \partial^{n}_\kappa P^\dagger(k,t)_m \mod p^{m+1}$.
\end{thm}

\proof
Let $P^\dagger(k,t)=\sum_{i=0}^\infty b_i(k)t^i$. In the sequel we fix a choice of $i$ and write $b(k)=b_i(k)$. Then, for $s_m=(p-1)p^m$ it is enough to show
\[
\sum_{j=0}^n(-1)^j{n\,\choose j} b(k + (n-j)ps_m)=p^n \sum_{j=0}^n(-1)^j{n\,\choose j} b(k + (n-j)s_m) \mod p^{(n+1)(m+1)}
\]
We define the difference functions $\delta_\nu(b,k)$ recursively by
\[
\begin{aligned}
\delta_1(b,k)&:=b(k+s_m)-b(k),\\
\delta_\nu(b,k)&:=\delta_{\nu-1}(b,k+s_m)-\delta_{\nu-1}(b,k) \quad \text{for } \nu\geq 2.
\end{aligned}
\]
By \cite[Theorem 2]{gm}, we have $\delta_\nu(b,k)\equiv 0 \mod p^{\nu(m+1)}$. Thus, rewriting the above equation, it suffices to show that
\[
\sum_{j=0}^n(-1)^j{n\,\choose j} b(k + (n-j)ps_m)-p^n \sum_{j=0}^n(-1)^j{n\,\choose j} b(k + (n-j)s_m)
\]
is a $\ZZ$-linear combination of $\delta_\nu(b,k)$ for $\nu\geq n+1$. By definition, we have
\[
\delta_\nu(b,k)=\sum_{j=0}^\nu {\nu\,\choose j}(-1)^j b(k+(\nu-j)s_m).
\]
Let $X$ be a variable and write
\[
(X-1)^\nu=\sum_{j=0}^\nu {\nu\,\choose j}(-1)^j X^{\nu-j}.
\]
Then, after substituting $X^{l}$ for $b(k+ls_m)$, we just need to prove that the polynomial $R(X)=(X^p-1)^n-p^n(X-1)^n$ is a $\ZZ$-linear combination of $(X-1)^\nu$ for $\nu\geq n+1$ or equivalently, that $R(X)$ vanishes to order at least $n+1$ at $X=1$. Clearly, $R(X)$ vanishes to order at least $n$ at $X=1$. Moreover, we have
\[
\frac{R(X)}{(X-1)^n}=\left(\frac{X^p-1}{X-1}\right)^n-p^n=\left(1+X+\dots+X^{p-1}\right)^n-p^n.
\]
and by specializing to $X=1$, we see that $R(X)$ vanishes to order at least $n+1$. This concludes the proof of the first part of the theorem. The second part is an immediate consequence, since the higher derivatives of every function given by a power series are approximated by finite differences on a small neighbourhood.\endproof

\section{Computations}
\label{sec6}
\subsection{Overview of the code}
The algorithm implemented computes the polynomial $\mathcal{Q}_{p,k}^{\epsilon}$, defined in Section~\ref{sec3}.
The input of the algorithm are a prime $p$, a positive integer $N$ coprime to $p$, an even integer $k$ and a positive integer $m$, the desired precision. 

The code works in a straightforward manner following the steps below:
\begin{enumerate}[Step 1:]
\item Find the Atkin-Lehner decomposition of the space. We will decompose the space according to the eigenvalues of the Atkin-Lehner operator $W_p$, since we will work on each $\mathcal{S}_k(\Gamma_0(pN))^{p-\mathrm{new}, \epsilon}$ separately, where $\epsilon$ is one of such eigenvalues. In this step, we actually only compute the dimension $d_\epsilon$ of such eigenspaces. As shown in Section~\ref{sec4}, this computation can be done via Corollary~\ref{formula_dim}, Equation~(1.6) in \cite{mar} and by recursion.
\item Compute evaluations of the reverse characteristic series of the $U_p$ operator on overconvergent modular forms at the weights $\{k,k+s_m,\dots,k+d_\epsilon s_m\}$ to precision $(d_\epsilon+1) m+1$, where $s_m=(p-1)p^m$. We performed this step using algorithms of Lauder \cite{lau} and Vonk \cite{vonk}, see Remark \ref{algo}. 
\item For each subspace $\mathcal{S}_k(\Gamma_0(pN))^{p-\mathrm{new}, \epsilon}$ and $1\leq n \leq d_\epsilon$, compute $\partial_\bullet^n P^\dagger(k,t)$, the partial derivative with respect to $\bullet$ of $P^\dagger(\kappa,t)$ at the point $(k,t)$ viewed as an element in $\ZZ_p\{\!\{t\}\!\}$. By Theorem~\ref{prec_deriv}, Step 2 provides all data needed for the derivative in $\kappa$, which is the relevant piece. 
\item Build the quotient $\mathcal{Q}^\epsilon_{p,k}$ of both derivatives, if possible. If the precision is too low, this step cannot be performed. In this case, or if the resulting precision of the coefficients $c_n$ as in Theorem~\ref{poly} is less than $m$, restart Step 2 with $m$ replaced by $m+10$.

\end{enumerate}

Some examples of the output are presented in the next subsection. Building from the output of the algorithm, we analyzed the distribution of (the slopes of) $\mathcal{L}$-invariants for increasing weight. We tabled the results obtained in Subsection~\ref{sectable} and we formulated some speculations, based on the data collected, in Subsection~\ref{secspec}.

\begin{rem}
It should be pointed out that, once the reverse characteristic series of $U_p$ is obtained, our algorithm is independent of the work of Lauder and Vonk, and can thus be combined with different methods of computing the series. The main bottleneck of our computations is the $p$-adic precision to which the coefficients $c_n$ of $\mathcal{Q}_{p,k}^{\epsilon}$ have to be computed. As shown in Theorem \ref{prec_deriv}, this precision depends on the dimension $d_\epsilon$ of the relevant space of modular forms. This leads to work with very large matrices of $U_p$ on overconvergent modular forms in the algorithms of Lauder and Vonk, making even the computation of the characteristic series rather time-consuming.
\end{rem}

\subsection{Examples}
All of the following examples were computed using an Intel$^{\tiny{\textregistered}}$
Xeon$^{\tiny{\textregistered}}$ E5-2650 v4 CPU processor with 512 GB of RAM memory.
\begin{ex}
Let us consider the space $\mathcal{S}_4(\Gamma_0(6))$. We have
\[\dim\mathcal{S}_4(\Gamma_0(6))^{p-\text{new},\epsilon} > 0 \quad \text{for} \quad  (p,\varepsilon)\in\{(2,1), (3, 1)\}.\]
In both cases the space is one-dimensional, i.e. the polynomial $\mathcal{Q}^\varepsilon_{p,4}(X)\in \mathbb{Q}_p[X]$ is given by a linear factor with the $\mathcal{L}$-invariant of the unique newform in $\mathcal{S}_4(\Gamma_0(6))^{p-\text{new}}$ as a zero. Modulo $2^{21}$, respectively  $3^{21}$, we have
\begin{align*}
\mathcal{Q}^{+1}_{2,4}(X) &= X + 94387\cdot2 \\
&= X - (2 + 2^3 + 2^4 + 2^7 + 2^9 + 2^{10} + 2^{11} + 2^{12} + 2^{16} + 2^{18} + 2^{19} + 2^{20}), \\
\mathcal{Q}^{+1}_{3,4}(X) &= X - 41502709 \\
&= X - (1 + 3^2 + 2\cdot3^7 + 3^8 + 2\cdot3^9 +2\cdot3^{13} + 2\cdot3^{14} + 2\cdot3^{15}).
\end{align*}
The latter example was also computed in \cite{tei90} and in \cite{pg}. The duration of the above computations was $1.8$ and $5.6$ seconds respectively.
\end{ex}
\begin{ex}
Let $p=7, N=11$ and $k=2$. The space $\mathcal{S}_{2}(\Gamma_0(77))^{7-\text{new},\epsilon}$ is two-dimensional for $\varepsilon = +1$ and three-dimensional for $\varepsilon = -1$. The coefficients of the following polynomials have been reduced modulo $7^{21}$.
\begin{align*}
\mathcal{Q}^{+1}_{7,2}(X) &= X^2 + 526982521374955003\cdot 7\cdot X + 192454376115114681\cdot 7^2 \\
\mathcal{Q}^{-1}_{7,2}(X) &= X^3 + 222104449450143105\cdot 7\cdot X^2 \\& \quad - 103688086480269397\cdot 7^2\cdot X - 39596022807935252 \cdot 7^3 
\end{align*}
The computation required approximately 2 hours. The slopes can be computed to be $[1_2]$ and $[1_3]$ respectively.  

In this case, we have  $\mathcal{S}_{2}(\Gamma_0(77))^{7-\text{new},\epsilon}= \mathcal{S}_{2}(\Gamma_0(77))^{\text{new},\epsilon}$ for $\epsilon=\pm 1$. The space $\mathcal{S}_{2}(\Gamma_0(77))^{\text{new},+1}$ is spanned by two rational newforms $f_1$ and $f_2$. These correspond to the isogeny classes of elliptic curves with Cremona labels \texttt{77a} and \texttt{77c} and their $\mathcal{L}$-invariants are given by
\begin{align*}
\mathcal{L}_7(f_1)&=2\cdot7 + 3\cdot 7^2 + 4\cdot 7^3 + 7^4 + 7^5 + 4\cdot 7^7 + 7^9 + O(7^{10}),
\\
\mathcal{L}_7(f_2)&=4\cdot 7 + 2\cdot 7^2 + 6\cdot 7^3 + 5\cdot 7^4 + 6\cdot 7^5 + 6\cdot 7^6 + 3\cdot 7^7 + 3\cdot 7^8 + 7^9 + O(7^{10}).
\end{align*}
The $\mathcal{L}$-invariants attached to elliptic curves can be computed, for example, by a build-in method in terms of Tate-parameters in \texttt{sage}. In the two cases above, one first has to apply a quadratic twist, as the curves do not have split-multiplicative reduction at $7$. \par
The space $\mathcal{S}_{2}(\Gamma_0(77))^{\text{new},-1}$ is spanned by a rational newform $f$ and two newforms $g_1$ and $g_2$ with coefficient field $K=\QQ(\sqrt{5})$. The newform $f$ corresponds to the isogeny class of elliptic curves with Cremona label \texttt{77b}, which has split multiplicative reduction at $7$. Its $\mathcal{L}$-invariant is given by
\[
\mathcal{L}_7(f)=3\cdot 7 + 7^2 + 2\cdot 7^3 + 5\cdot 7^4 + 4\cdot 7^5 + 7^6 + 4\cdot 7^7 + 7^8 + 3\cdot 7^9 + O(7^{10}).
\]
If we divide the polynomial $ \mathcal{Q}^{-1}_{7,2}(X) $ by $(X-\mathcal{L}_7(f))$, we thus obtain (modulo $7^{11}$)
\begin{align*}
(X-\mathcal{L}_7(g_1))\cdot (X-\mathcal{L}_7(g_2))= X^2+ 225931960\cdot7\cdot X+
21342634\cdot 7^2.
\end{align*}
This polynomial is in fact irreducible over $\QQ_7$, showing that the $\mathcal{L}$-invariant $\mathcal{L}_7(g_1)$ generates the quadratic unramified extension of $\QQ_7$, and similarly for $\mathcal{L}_7(g_2)$. We have computed the coefficient fields $\QQ_{g_1,7}$ and $\QQ_{g_2,7}$  and obtained that $\QQ_{g_i,7}=\QQ_7(\mathcal{L}_7(g_i))$ for $i=1,2$. This is a first example of forms of  higher level ($N=11$), where Chenevier's description (for $N=1$) of the local coefficient field, mentioned in the introduction, still holds.
\end{ex}

\subsection{Observations collected}
\label{secspec}
Given a positive even integer $k$, we denote by $\nu_\mathcal{L}^{\epsilon}(k,p,N)$ the finite sequence of slopes of the $p$-adic $\mathcal{L}$-invariants attached to forms in $\mathcal{S}_k(\Gamma_0(pN))^{p-\mathrm{new},\epsilon}$ ordered in decreasing order, where $\epsilon$ is an eigenvalue of the Atkin-Lehner operator $W_p$ acting on $\mathcal{S}_k(\Gamma_0(pN))^{p-\mathrm{new}}$. Note that this differs from the tables in \cite{pg}, where the space is decomposed with respect to the Atkin-Lehner involution at $N$. 

The observations collected are based on the data and on the tables presented in the next section.
\smallskip

We begin with the case $p=2$, that was extensively studied in \cite{pg}. There, several relations between slopes of different level and weights were conjectured for levels $N=3,5,7$. Table \ref{table:1} provides data for the case $N=1$, which is not accessible by the methods in \cite{pg}. The observations lead us to the following analogous conjecture.
\begin{con}[$p=2$]~
\label{con1}
\begin{itemize}
\item[(a)]For $k\in 2+4\ZZ$, $k\geq 10$ and $\epsilon\in\{\pm1\}$, we have
\[
\nu_\mathcal{L}^\epsilon(k,2,1)=\nu_\mathcal{L}^{-\epsilon}(k+6,2,1).
\]
\item[(b)] For every even integer $k$, the final $\min\{d_{+1},d_{-1}\}$ slopes in $\nu_\mathcal{L}^{+}(k,2,1)$ and $\nu_\mathcal{L}^{-}(k,2,1)$ agree.
\end{itemize}
\end{con}
Table \ref{table:1} verifies this statement up to $k=70$ computationally.  Our observations is in line with similar observations in \cite{pg}. The table also provides the missing data for \cite[Conjecture 5.6 (ii)]{pg} stating that the slopes appearing in level $7$ are the union of (two copies) of the slopes in levels $1$ and $3$.\par 
\vspace{\baselineskip}
The case $p=3$ and $N=1$ is treated in Table \ref{table:2}. Again we observe relations between various different slopes that do not yet have any theoretical interpretation. We collect our observations in the following conjecture.
\begin{con}[$p=3$]~
\label{con2}
\begin{itemize}
\item[(a)] For $k\in 2+6\ZZ$, $k\geq 8$ and $\epsilon\in\{\pm1\}$, we have
\[
\nu_\mathcal{L}^\epsilon(k,3,1)=\nu_\mathcal{L}^{\epsilon}(k+4,3,1).
\]
\item[(b)] For every even integer $k$, the final $\min\{d_{+1},d_{-1}\}$ slopes in $\nu_\mathcal{L}^{+}(k,3,1)$ and $\nu_\mathcal{L}^{-}(k,3,1)$ agree.
\end{itemize}
\end{con}
Interestingly, the numbers $4$ and $6$ from Conjecture \ref{con1} (a) are essentially switched here: we have relations between $k$ and $k+4$ when considering $k \mod 6$. The same is true for the Atkin-Lehner sign, since here we always have relations between slopes of the same sign. While Conjecture \ref{con2} is very similar to Conjecture \ref{con1}, we should point out that here there are slopes that do not have (obvious) relations to other slopes. These are the slopes for $k\in 4+6\ZZ$. We still observe that for example the slopes for $k=10$ and $k=16$ only differ by $-1$. Similarly, the slopes for $k=22$ and $k=28$ are very similar. However, we are not able to formulate a precise conjecture.\par
\vspace{\baselineskip}
The data in Tables~\ref{table:3}, \ref{table:4} and \ref{table:5} shows various relations between slopes of different weights supporting conjectural connections between them, analogous to the conjectures presented above. 
Moreover, after possibly removing oldforms, the analogous observation as in part (b) of the Conjectures~\ref{con1} and~\ref{con2} is still present in the data presented in Tables~\ref{table:6} and \ref{table:7}. We are able to say even more, namely that observation (b) is a purely local phenomenon. Let for example $p=13$ and $k=10$, the corresponding space of modular forms consists of two Galois orbits of dimensions $4$ and $5$. Each of the orbits makes up one of the two columns in Table \ref{table:5}. Thus, the pairs in observation (b) arise from different orbits. Moreover, within each orbit (aside from the first slope) all slopes are distinct.

The data computed for $p=2$ and $N=3$ matches with the one presented in \cite{pg} and we observe again a relation between the slopes, see \cite[Conjecture~5.2]{pg}. In all cases, it is an interesting question to analyze the growth of the slopes with the weight $k$ further. However, at this stage, the collected data is not sufficient to make precise claims. We observe that, in general, almost all slopes are negative, thus providing evidence for the conjectureral statement  that all $\mathcal{L}$-invariants are non-zero in Remark \ref{nonzero}. It would also be interesting to separate the slopes with respect to $\mathrm{mod}\ p$ eigensystems. This is however not possible with our method. As this is only relevant for larger $p$ and $N$, this question becomes more interesting once more data in these cases is available.
\definecolor{bluegray}{rgb}{0.4, 0.6, 0.8}

\subsection{Tables}
\label{sectable}
We keep the notation of subsection~\ref{secspec}. Note that in the case $N>1$ our tables can include oldforms. Their slopes are indicated by 
{\color{bluegray} blue} font. 

\begin{table}[!htb]
\centering
\begin{tabular}{c|c|r|r}
$k$ & $d$ & $\nu_\mathcal{L}^+(k,2,1)$ &  $\nu_\mathcal{L}^-(k,2,1)$  \\
\hline
$8$ & $1$ &  \texttt{0} & \\
\hline
$10$ & $1$ & & \texttt{-1} \\
\hline
$12$ & $0$ & & \\ 
\hline
$14$ & $2$ &  \texttt{-4} &  \texttt{-4} \\
\hline
$16$ & $1$ &  \texttt{-1} & \\
\hline
$18$ & $1$ & &  \texttt{-2} \\
\hline
$20$ & $2$ &  \texttt{-4} &  \texttt{-4} \\
\hline
$22$ & $2$ &  \texttt{-6} &  \texttt{-6} \\
\hline
$24$ & $1$ &  \texttt{-2} & \\
\hline
$26$ & $3$ &  \texttt{-7} &  \texttt{-2, ~-7} \\
\hline
$28$ & $2$ &  \texttt{-6}&  \texttt{-6} \\
\hline
$30$ & $2$ &  \texttt{-5} &  \texttt{-5} \\
\hline
$32$ & $3$ &  \texttt{-2, ~-7} &  \texttt{-7} \\
\hline
$34$ & $3$ &  \texttt{-7} &  \texttt{-3, ~-7} \\
\hline
$36$ & $2$ &  \texttt{-5} &  \texttt{-5} \\
\hline
$38$ & $4$ &  \texttt{-6, -10} &  \texttt{-6, -10} 
\end{tabular}\quad \quad
\centering 
\begin{tabular}{c|c|r|r}
$k$ & $d$ & $\nu_\mathcal{L}^+(k,2,1)$ &  $\nu_\mathcal{L}^-(k,2,1)$  \\
\hline
$40$ & $3$ &  \texttt{-3,~ -7} & \texttt{-7} \\
\hline
$42$ & $3$ &  \texttt{-9} &  \texttt{-3,~ -9} \\
\hline
$44$ & $4$ &  \texttt{-6, -10} &  \texttt{-6, -10} \\
\hline
$46$ & $4$ &  \texttt{-4, -11} &  \texttt{-4, -11} \\
\hline
$48$ & $3$ &  \texttt{-3,~ -9} &  \texttt{-9} \\
\hline
$50$ & $5$ &  \texttt{-9, -12} &  \texttt{-3, ~-9, -12} \\
\hline
$52$ & $4$ &  \texttt{-4, -11} &  \texttt{-4, -11} \\
\hline
$54$ & $4$ &  \texttt{-7, -10} &  \texttt{-7, -10} \\
\hline
$56$ & $5$ &  \texttt{-3,~ -9, -12} &  \texttt{-9, -12} \\
\hline
$58$ & $5$ &  \texttt{-8, -12} &  \texttt{-3,~ -8, -12} \\
\hline
$60$ & $4$ &  \texttt{-7, -10} &  \texttt{-7, -10} \\
\hline
$62$ & $6$ &  \texttt{-6, -12, -12} &  \texttt{-6, -12, -12} \\
\hline
$64$ & $5$ &  \texttt{-3,~ -8, -12} &  \texttt{-8, -12} \\
\hline
$66$ & $5$ &  \texttt{-8, -12} &  \texttt{-4, ~-8, -12} \\
\hline
$68$ & $6$ &  \texttt{-6, -12, -12} & \texttt{-6, -12, -12} \\
\hline
$70$ & $6$ & \texttt{-7, -10, -16}  & \texttt{-7, -10, -16}   \\
\end{tabular}
\medskip
\caption{\small $p=2, N=1$}
\label{table:1}
\end{table}

\begin{table}[!htb]
\centering
\begin{tabular}{c|c|r|r}
$k$ & $d$ & $\nu_\mathcal{L}^+(k,3,1)$ &  $\nu_\mathcal{L}^-(k,3,1)$  \\
\hline
$6$& $1$&  &    \texttt{1} \\
\hline
$8$& $1$&  \texttt{-1} &     \\
\hline
$10$& $2$&  \texttt{-2} &     \texttt{-2} \\
\hline
$12$& $1$&  \texttt{-1} &     \\
\hline
$14$& $3$&  \texttt{-4} &    \texttt{0,~-4} \\
\hline
$16$& $2$&  \texttt{-3} &     \texttt{-3} \\
\hline
$18$& $3$&  \texttt{-4} &    \texttt{0, ~-4} \\
\hline
$20$& $3$&  \texttt{-2,~ -4} &    \texttt{ -4} \\
\hline
$22$& $4$&  \texttt{-3,~ -6} &    \texttt{-3, ~-6} \\
\hline
$24$& $3$&  \texttt{-2,~ -4} &    \texttt{-4} \\
\hline
$26$& $5$&  \texttt{-5,~ -7} &     \texttt{-2, ~-5,~ -7} \\
\hline
$28$& $4$&  \texttt{-1,~ -6} &     \texttt{-2, ~-6} \\
\hline
$30$& $5$&  \texttt{-5,~ -7} &     \texttt{-2, ~-5, ~-7} \\
\hline
$32$& $5$&  \texttt{-2, ~-4,~ -9} &     \texttt{-4, ~-9} \\
\hline
$34$& $6$&  \texttt{-3, ~-6, -10} &     \texttt{-3, ~-6, -10} \\
\hline
$36$& $5$&  \texttt{-2, ~-4,~ -9} &     \texttt{-4,~ -9}\\
\hline
$38$& $7$&  \texttt{-4, ~-9, -11} &     \texttt{-1,~ -4, ~-9, -11}\\
\hline
$40$& $6$&  \texttt{-3, ~-7, -10} &    \texttt{-3, ~-7, -10} \\
\hline
$42$& $7$&  \texttt{-4, ~-9, -11} & \texttt{-1,~ -4,~ -9, -11}\\
\hline
$44$& $7$&  \texttt{0, ~-6,~ -8, -11} &   \texttt{-6,~ -8, -11}\\
\hline
$46$& $8$&  \texttt{-2, ~-7, -11, -11}&    \texttt{-2, ~-7, -11, -11} \\
\hline
$48$& $7$&  \texttt{0, ~-6, ~-8, -11} &   \texttt{-6, ~-8, -11}\\
\hline
$50$& $9$& \texttt{-5, ~-9, -11, -13}  &  \texttt{-1, ~-5, ~-9, -11, -13}
\end{tabular}
\medskip
\caption{\small $p=3, N=1$}
\label{table:2}
\end{table}

\begin{table}[!htb]
\centering
\begin{tabular}{c|c|r|r}
$k$ & $d$ & $\nu_\mathcal{L}^+(k,5,1)$ &  $\nu_\mathcal{L}^-(k,5,1)$  \\
\hline
4& 1 &  \texttt{0} & \\
\hline
6& 1 & &  \texttt{0} \\
\hline
8& 3 &  \texttt{0, ~-2} &  \texttt{-2} \\
\hline
10& 3 &  \texttt{-2} &  \texttt{2, ~-2} \\
\hline
12& 3 &  \texttt{-1,~ -2} &  \texttt{-2} \\
\hline
14& 5 & \texttt{-2,~ -4} &  \texttt{-1,~ -2,~ -4} \\
\hline
16& 5 &  \texttt{-1, ~-3, ~-4} &  \texttt{-3, ~-4} \\
\hline
18& 5 &  \texttt{-2, ~-4} &  \texttt{-1, ~-2, -~4} \\
\hline
20& 7 &  \texttt{-1, ~-2,~ -5, ~-6} &  \texttt{-2, ~-5,~ -6} \\
\hline
22& 7 &  \texttt{-2, ~-5,~ -6} &  \texttt{-1, ~-2,~ -5,~ -6} \\
\hline
24& 7 &  \texttt{-1, ~-2, ~-4, ~-7} &  \texttt{-2,~ -4,~ -7} \\
\hline
26& 9 &  \texttt{-3,~ -4,~ -7,~ -8} &  \texttt{-1, ~-3, ~-4, ~-7, ~-8} \\
\hline
28& 9 &  \texttt{-1, ~-2, ~-4, ~-7, -10} &  \texttt{-2, ~-4, ~-7, -10} 
\end{tabular}
\bigskip
\caption{\small $p=5, N=1$}
\label{table:3}
\end{table}

\begin{table}[!htb]
\centering
\begin{tabular}{c|c|r|r}
$k$ & $d$ & $\nu_\mathcal{L}^+(k,7,1)$ &  $\nu_\mathcal{L}^-(k,7,1)$  \\
\hline
4& 1 &  \texttt{0} &  \\
\hline
6& 3 &  \texttt{-1}& \texttt{ 0, -1} \\
\hline
8& 3 &  \texttt{0, -1} & \texttt{ -1}\\
\hline
10& 5 &  \texttt{-1, -3} &  \texttt{ 0, -1, -3} \\
\hline
12& 5 &  \texttt{0, -2, -3} & \texttt{ -2, -3}\\
\hline
14& 7 &  \texttt{-2, -3, -4} &  \texttt{2, -2, -3, -4} \\
\hline
16& 7 &  \texttt{-1, -2, -3, -4}  &  \texttt{ -2, -3, -4} \\
\hline
18& 9 &  \texttt{-2, -3, -5, -6} & \texttt{  -1, -2, -3, -5, -6} \\
\hline
20& 9 &  \texttt{0, -1, -3, -5, -6} &  \texttt{ -1, -3, -5, -6 }\\
\hline
22& 11 &  \texttt{-1, -4, -5, -6, -7} &   \texttt{-1, -1, -4, -5, -6, -7} 
\end{tabular}
\medskip
\caption{\small$p=7, N=1$}
\label{table:4}
\end{table}

\begin{table}[!htb]
\centering
\begin{tabular}{c|c|r|r}
$k$ & $d$ & $\nu_\mathcal{L}^+(k,13,1)$ &  $\nu_\mathcal{L}^-(k,13,1)$  \\
\hline
4 & 3 &  \texttt{0, ~0}  & \texttt{0} \\
\hline
6 & 5 &  \texttt{0, -1} &  \texttt{0, ~0, -1} \\
\hline
8 & 7 &  \texttt{0, ~0, -1, -2}  & \texttt{0, -1, -2} \\
\hline
10 & 9 &  \texttt{0, -1, -2, -3}&  \texttt{0, ~0, -1, -2, -3} \\
\hline
12 & 11 & \texttt{0,~ 0, -1, -2, -3, -4} & \texttt{0, -1, -2, -3, -4} \\
\hline
14 & 13 &  \texttt{0, -1, -2, -3, -4, -5} & \texttt{0, 0, -1, -2, -3, -4, -5} 
\end{tabular}
\medskip
\caption{\small$p=13$, $N=1$}
\label{table:5}
\end{table}

\begin{table}[!htb]
\centering
\begin{tabular}{c|c|r|r}
$k$ & $d$ & $\nu_\mathcal{L}^+(k,2,3)$ &  $\nu_\mathcal{L}^-(k,2,3)$  \\
\hline
$4$     &$1$& \texttt{1} & \\
\hline
$6$     &$1$&  & \texttt{0}  \\
\hline
$8$     &$3$& \texttt{{\color{bluegray}0}, {\color{bluegray}0}} & \texttt{-1} \\ 
\hline
$10$     &$3$& \texttt{0} & \texttt{{\color{bluegray}-1}, {\color{bluegray} -1}} \\ 
\hline
$12$     &$3$& \texttt{-1, -4} & \texttt{-4} \\ 
\hline
$14$     &$5$& \texttt{{\color{bluegray}-4}, {\color{bluegray}-4}} & \texttt{-1, {\color{bluegray}-4}, {\color{bluegray}-4}}\\ 
\hline
$16$     &$5$& \texttt{{\color{bluegray}-1}, {\color{bluegray}-1}, -4} & \texttt{-2, -4} \\ 
\hline
$18$     &$5$& \texttt{-1, -4} & \texttt{{\color{bluegray}-2}, {\color{bluegray}-2}, -4} \\ 
\hline
$20$     &$7$& \texttt{-2, {\color{bluegray}-4}, {\color{bluegray}-4}, -6} & \texttt{{\color{bluegray}-4}, {\color{bluegray}-4}, -6} \\ 
\hline
$22$     &$7$& \texttt{-4, {\color{bluegray}-6}, {\color{bluegray}-6}} & \texttt{-2, -4, {\color{bluegray}-6}, {\color{bluegray}-6}} \\ 
\hline
$24$     &$7$& \texttt{{\color{bluegray}-2}, {\color{bluegray}-2}, -6, -7} & \texttt{-2, -6, -7} 
\end{tabular}
\medskip
\caption{\small $p=2, N=3$}
\label{table:6}
\end{table}

\begin{table}[!htb]
\centering
\begin{tabular}{c|c|r|r}
$k$ & $d$ & $\nu_\mathcal{L}^+(k,3,2)$ &  $\nu_\mathcal{L}^-(k,3,2)$  \\
\hline
$4$     &$1$& \texttt{0} & \\
\hline
$6$     &$3$& \texttt{-1} & \texttt{{\color{bluegray}1},~ {\color{bluegray}1}}  \\
\hline
$8$     &$3$& \texttt{{\color{bluegray}-1}, {\color{bluegray}-1}} & \texttt{0} \\ 
\hline
$10$     &$5$& \texttt{{\color{bluegray}-2}, {\color{bluegray}-2}} & \texttt{-1, {\color{bluegray}-2}, {\color{bluegray}-2}} \\ 
\hline
$12$     &$5$& \texttt{{\color{bluegray}-1} ,{\color{bluegray}-1}, -4} & \texttt{0, -4} \\ 
\hline
$14$     &$7$& \texttt{-1, {\color{bluegray}-4}, {\color{bluegray}-4}} & \texttt{{\color{bluegray}0}, ~{\color{bluegray}0}, {\color{bluegray}-4}, {\color{bluegray}-4}}\\ 
\hline
$16$     &$7$& \texttt{1, {\color{bluegray}-3}, {\color{bluegray}-3}, -4} & \texttt{{\color{bluegray}-3}, {\color{bluegray}-3}, -4} \\ 
\hline
$18$     &$9$& \texttt{-2, -4, {\color{bluegray} -4}, {\color{bluegray}-4}} & \texttt{{\color{bluegray}0}, ~{\color{bluegray}0}, -4, {\color{bluegray}-4}, {\color{bluegray}-4}} \\ 
\hline
$20$     &$9$& \texttt{{\color{bluegray}-2}, {\color{bluegray}-2}, -4, {\color{bluegray} -4}, {\color{bluegray}-4}} & \texttt{0, -4, {\color{bluegray} -4}, {\color{bluegray}-4}} \\ 
\hline
$22$     &$11$& \texttt{{\color{bluegray}-3}, {\color{bluegray}-3}, -4, {\color{bluegray}-6}, {\color{bluegray}-6}} & \texttt{-2, {\color{bluegray}-3}, {\color{bluegray}-3}, -4, {\color{bluegray}-6}, {\color{bluegray}-6}}\\
\multicolumn{4}{c}{}
\end{tabular}
\medskip
\caption{\small $p=3, N=2$}
\label{table:7}
\end{table}


\clearpage
\bibliographystyle{annotate}
\bibliography{bibtexLinv}
\end{document}